\documentclass[12pt]{amsart}
\usepackage{amssymb,latexsym,amsmath,enumerate,amsthm,amscd,mathrsfs,yfonts,hyperref}
\usepackage{fullpage}
\input xy
\xyoption{all}

\newcommand{\ZZ}{\mathbb Z}
\newcommand{\QQ}{\mathbb Q}
\newcommand{\CC}{\mathbb C}

\newcommand{\cpt}{\mathbb K}

\def\bK{\mathbf K}
\def\F{\mathbf F}

\def\BC{\textup{BC}}

\def\Cone{\textup{Cone}}

\def\cF{\mathtt{Fin}}
\def\id{\mathrm{id}}

\def\Or{\textup{Or}}
\def\NK{\textup{NK}}

\def\prot{\hat{\otimes}}

\def\K{\textup{K}}
\def\SS{\mathbb{S}}

\def\alg{\textup{alg}}
\def\L{\textup{L}}

\def\S{\mathcal{S}}

\def\Tor{\textup{Tor}}

\def\KT{\textup{K}^{\textup{top}}}

\def\top{\textup{top}}

\def\KK{\textup{KK}}
\def\E{\mathbf{E}}

\def\H{\textup{H}}

\def\DL{\textup{DL}}

\def\cN{\mathcal N}
\def\1{\bf{1}}

\def\VC{\mathtt{VCyc}}

\newcommand{\map}{\rightarrow}

\newcommand{\beq}{\begin{eqnarray}}
\newcommand{\beqn}{\begin{eqnarray*}}
\newcommand{\eeq}{\end{eqnarray}}
\newcommand{\eeqn}{\end{eqnarray*}}

\theoremstyle{definition}
\newtheorem{thm}{Theorem}[section]
\theoremstyle{definition}

\theoremstyle{definition}
\newtheorem{lem}[thm]{Lemma}
\theoremstyle{definition}
\newtheorem{prop}[thm]{Proposition}
\theoremstyle{definition}

\theoremstyle{definition}
\newtheorem{conj}[thm]{Conjecture}
\theoremstyle{definition}

\theoremstyle{definition}
\newtheorem{rem}[thm]{Remark}
\theoremstyle{definition}
\newtheorem{defn}[thm]{Definition}

\theoremstyle{definition}
\newtheorem*{theorem}{Theorem}

\begin{document}

\title{Assembly maps with coefficients in topological algebras and the integral $\K$-theoretic Novikov conjecture}
\author{Snigdhayan Mahanta}
\email{s.mahanta@uni-muenster.de}
\address{Mathematical Institute, University of Muenster, Einsteinstrasse 62, 48149 Muenster, Germany.}

\thanks{This research was supported under Australian Research Council's Discovery Projects funding scheme (project number DP0878184) and the Deutsche Forschungsgemeinschaft (SFB 878).}

\begin{abstract}
We prove that any countable discrete and torsion free subgroup of a general linear group over an arbitrary field or a similar subgroup of an almost connected Lie group satisfies the integral algebraic $\K$-theoretic (split) Novikov conjecture over $\cpt$ and $\S$, where $\cpt$ denotes the $C^*$-algebra of compact operators and $\S$ denotes the algebra of Schatten class operators. We introduce assembly maps with finite coefficients and under an additional hypothesis, we prove that such a group also satisfies the algebraic $\K$-theoretic Novikov conjecture over $\overline{\QQ}$  and $\CC$ with finite coefficients. For all torsion free Gromov hyperbolic groups $G$, we demonstrate that the canonical algebra homomorphism $\cpt[G]\map C^*_r(G)\prot\cpt$ induces an isomorphism between their algebraic $\K$-theory groups. 
\end{abstract}

\maketitle

\begin{center}
{\bf Introduction}
\end{center}

For any discrete group $G$ and a unital ring $R$ the algebraic $\K$-theoretic Novikov conjecture for $G$ over $R$ asserts that a canonically defined Loday assembly map $$(\mu^\L_R)_*:\H_*(BG;\bK_R)\map\K_*(R[G])$$ is rationally injective. Here $\bK_R$ denotes the nonconnective algebraic $\K$-theory spectrum of $R$. In fact, there is a map of spectra $\mu^\L_R: BG_+\wedge \bK_R\map\bK_{R[G]}$, which induces $(\mu^\L_R)_*$. The stronger integral $\K$-theoretic (resp. split $\K$-theoretic) Novikov conjecture asserts that $(\mu^\L_R)_*$ is injective (resp. split injective). Using standard excision arguments and the fact that $\H$-unital $\QQ$-algebras in the sense of \cite{Wod} satisfy excision in algebraic $\K$-theory \cite{SusWod2}, the Loday assembly map can be extended to $\H$-unital coefficient $\QQ$-algebras $R$. Throughout this article the term $\K$-theory without any adjective will refer to algebraic $\K$-theory and will be denoted by $\K_*$. Let us introduce a few more notations.

\noindent
{\bf Notations and conventions:} 
\begin{description}
\item[$(\mu^\L_R)_*$] Loday assembly map with coefficients in $R$,
\item[$(\mu^{\DL}_R)_*$] Davis--L{\"u}ck assembly map with coefficients in $R$,
\item[$(\mu^{\BC}_A)_*$] Baum--Connes assembly map with coefficients in $A$ (separable $C^*$-algebra).
\end{description}

\noindent
Sometimes we are going to suppress the coefficient algebra $R$ or $A$ from the notation if it is the complex numbers, i.e., $\mu^?_*=(\mu^?_\CC)_*$, where $?=\L,\DL,\BC$. All groups are assumed to be countable and $\prot$ will denote the minimal $C^*$-tensor product. We are going to freely use the notations in \cite{DavLue}.

In this article we prove the following two results ({\em reduction principles}):

\begin{thm} \label{thm2} (see Theorem \ref{cpt} and Theorem \ref{overS})
If a countable discrete and torsion free group satisfies the integral $\KT$-theoretic Novikov (resp. the split $\KT$-theoretic Novikov) conjecture with $\CC$-coefficients, i.e., $\mu^\BC_*$ is injective (resp. split injective), then it satisfies the integral $\K$-theoretic Novikov (resp. the split $\K$-theoretic Novikov)  conjecture over $\cpt$ and $\S$, where $\cpt$ denotes the $C^*$-algebra of compact operators and $\S$ denotes the algebra of Schatten class operators.

\noindent
If $\mu^\BC_*$ is merely rationally injective, then so is $(\mu^\L_\cpt)_*$. (The rational injectivity of $(\mu^\L_\S)_*$ is known for all groups without any further hypothesis \cite{GuoliangNovikov}).
\end{thm}

\begin{thm} \label{thm1} (see Theorem \ref{overC})
Let $G$ be a countable discrete and torsion free group satisfying the split $\KT$-theoretic Novikov conjecture with $\CC$-coefficients, i.e., $\mu^\BC_*$ is split injective. Assume, in addition, that the $\ZZ/n$-homology of $BG$ is concentrated in even degrees. Then $G$ satisfies the $\K$-theoretic Novikov conjecture over $\overline{\QQ}$ and $\CC$ with finite $\ZZ/n$-coefficients.
\end{thm}

\begin{rem}
In the proof of the above Theorem \ref{thm1} one needs to ensure that the canonical map from the connective $\K$-homology of $BG$ to the nonconnective $\K$-homology of $BG$ (both with finite $\ZZ/n$-coefficients) is injective. The condition that the $\ZZ/n$-homology of $BG$ is concentrated in even degrees implies the injectvity of this map; however, it may hold under more general circumstances (see also Remark \ref{Gen}).
\end{rem}

Our strategy makes use of the Davis--L{\"u}ck unified perspective on the isomorphism conjectures. Some intricate analysis of assembly maps with finite coefficients, which we define in this article (see Definition \ref{FinAss}), is done invoking Suslin's work on algebraic $\K$-theory \cite{Suslin,SusLoc}. Along the way in Section \ref{SpecSeq} we produce K{\"u}nneth type spectral sequences using the machinery of \cite{EKMM} for computing the domains of the assembly maps. Although an extremely trivial case of these spectral sequences is used in this article, we hope that they will be of independent interest. We also observe that for any torsion free Gromov hyperbolic group $G$ the canonical algebra homomorphism $\cpt[G]\map C^*_r(G)\prot\cpt$ induces an isomorphism between their algebraic $\K$-theory groups (see Theorem \ref{mySW}). Let us mention that the split $\KT$-theoretic Novikov conjecture is known to be true in numerous examples. For instance, let $G$ be a countable discrete group with a proper left-invariant metric. Thanks to Skandalis--Tu--Yu we know that if $G$ admits a uniform embedding into a Hilbert space, then $G$ satisfies the split $\KT$-theoretic Novikov conjecture with coefficients in any separable $G$-$C^*$-algebra \cite{YuCoarseBC,SkaTu}. Guentner--Higson--Weinberger showed that if $G$ is a countable discrete subgroup of $\textup{GL}(n,F)$ for any field $F$ or of any almost connected Lie group, then $G$ satisfies the split $\KT$-theoretic Novikov conjecture with coefficients in any separable $G$-$C^*$-algebra \cite{GueHigWei}. As a consequence of the above-mentioned {\em reduction principles}, we arrive at the following interesting application:

\begin{theorem}
Any countable discrete and torsion free subgroup $G$ of a general linear group over an arbitrary field or a similar subgroup of an almost connected Lie group satisfies the split $\K$-theoretic Novikov conjecture over $\cpt$ and $\S$, i.e., the Loday assembly maps $$\text{$(\mu^\L_\cpt)_*:\H_*(BG;\bK_\cpt)\map\K_*(\cpt[G])$ and $(\mu^\L_\S)_*:\H_*(BG;\bK_\S)\map\K_*(\S[G])$}$$ are split injective. They also satisfy the $\K$-theoretic Novikov conjecture with finite coefficients over $\overline{\QQ}$ and $\CC$, i.e., the Loday assembly maps $(\mu^\L_{\overline{\QQ}})_*$ and $(\mu^\L_\CC)_*$ are injective with finite coefficients, provided $G$ satisfies the additional hypotheses of Theorem \ref{thm1}.

\noindent
For a countable discrete and torsion free group with a proper left-invariant metric, the above assertions continue to hold if the group admits a uniform embedding into a Hilbert space.
\end{theorem}

\noindent
The algebra of Schatten class operators $\S$ is very important from the viewpoint of higher index theory. The split injectivity result above can also be deduced for many groups (e.g., when $G$ is a discrete subgroup of an almost connected Lie group) with arbitrary coefficient algebra from an earlier work of Bartels--Rosenthal \cite{BarRos} (see also some related work of Ji \cite{JiNovikov}). Our result can be regarded as an application of the Baum--Connes conjecture \cite{BaumConnes}; more precisely, that of the (split) injectivity part of the assertion.

\vspace{3mm}
\noindent
{\bf Acknowledgements.} The author is extremely grateful to G. Yu for his comments on the first draft of this article. The author also wishes to thank P. Baum, R. J. Deeley, R. Meyer, H. Reich, J. M. Rosenberg, A. Valette and C. Westerland for helpful email correspondences. Finally the author is indebted to the anonymous referees for pointing out a mistake in the first draft and the constructive feedback.

\section{Some spectral sequences} \label{SpecSeq} For any $C^*$-algebra $A$, there is a symmetric spectrum (in the sense of \cite{HSS}) model of $\bK^\top_A$, which is, in addition, a (left) module spectrum over a (commutative) symmetric ring spectrum model of $\bK^\top_\CC$ (see Theorem B of \cite{JoaKHom}). Furthermore, there is a unit map from the sphere spectrum $\SS$ to $\bK^\top_\CC$, which is a homomorphism of commutative symmetric ring spectra. After passing to functorial cofibrant replacements (in the $\SS$-model structures \cite{ShiModel} {\em or the flat stable model structures as in Schwede's book on symmetric spectra \cite{SchwedeBook}}) on the categories of (left) module spectra over the symmetric ring spectra, we may assume that all spectra are cofibrant. Now apply the functorial left Quillen construction, which produces a (cofibrant) $\SS$-algebra (resp. $\SS$-module) from a (cofibrant) symmetric ring spectrum (resp. symmetric spectrum) as explained in \cite{SchSModSymSpec}. Thus we obtain a model of $\bK^\top_A$ as a (left) $\bK^\top_\CC$-module over an $\SS$-algebra model of $\bK^\top_\CC$, where all $\SS$-algebras (resp. $\SS$-modules) are cofibrant. For the details about $\SS$-algebras and $\SS$-modules the readers may refer to \cite{EKMM}. Now one may write $$BG_+\wedge\bK^\top_A\simeq (BG_+\wedge\bK^\top_\CC)\wedge_{\bK^\top_\CC} \bK^\top_A,$$ using a CW model of $BG$. If $R$ is a (cofibrant) commutative $\SS$-algebra and $M,N$ are (cofibrant) $R$-modules, then there is a strongly convergent natural (both in $M$ and $N$) spectral sequence (see Theorem 4.1 of \cite{EKMM})

\beq \label{mainSS}
 E^2_{p,q} =\Tor^{\pi_*(R)}_{p,q}(\pi_*(M),\pi_*(N))\Rightarrow \pi_{p+q} (M\wedge_R N).
 \eeq Here $p$ is the resolution degree of $M$ and $q$ is the internal degree of the graded modules whence it is a right half plane homological spectral sequence.

\begin{rem}
The symmetric spectra constructed in \cite{JoaKHom} take values in pointed simplicial sets, whereas $\SS$-modules are spectra valued in based spaces. However, it is known that there is a Quillen equivalence between the category of symmetric spectra valued in pointed simplicial sets and that of symmetric spectra valued in based spaces (see Section 18 of \cite{MMSS}).
\end{rem}

\noindent
Setting $R=\bK^\top_\CC$, $M=BG_+\wedge \bK^\top_\CC$ and $N=\bK^\top_A$ and observing that $\pi_*(BG_+\wedge\bK^\top_\CC)=\K^\top_*(BG)$ we get:

\begin{lem}
There is a right half plane homological strongly convergent natural spectral sequence
\beq \label{SS}
E^2_{p,q}=\Tor^{\ZZ[u,u^{-1}]}_{p,q}(\K^\top_*(BG),\K^\top_*(A))\Rightarrow \pi_{p+q}(BG_+\wedge\bK^\top_A),
\eeq where the degree of $u$ is $2$.
\end{lem}

\begin{prop} \label{domain}
There is an identification $$\K^\top_*(BG)\otimes\K^\top_*(\cpt)\cong\pi_*(BG_+\wedge\bK^\top_\cpt)=\H_*(BG;\bK^\top_\cpt),$$ which is natural in $G$.
\end{prop}

\begin{proof}
One knows that $\KT_*(\CC)\cong\KT_*(\cpt)\simeq\ZZ[u,u^{-1}]$ via the corner embedding $\CC\hookrightarrow\cpt$. The assertion is now evident from the above spectral sequence \eqref{SS}.
\end{proof}

\begin{rem}
Setting $R=\bK_\CC$, $M= BG_+\wedge\bK_\CC$ and $N=\bK_A$ in the spectral sequence \eqref{mainSS}, we get $$E^2_{p,q}=\Tor^{\K_*(\CC)}_{p,q}(\H_*(BG;\bK_\CC),\K_*(A))\Rightarrow \pi_{p+q}(BG_+\wedge\bK_A)=\H_{p+q}(BG;\bK_A).$$
\end{rem}

Apart from the standard Atiyah--Hirzebruch spectral sequences, these K{\"u}nneth type spectral sequences can potentially be useful for computational purposes for the domain of the Davis--L{\"u}ck assembly map in certain situations (compare \cite{RosSch1}).

\section{Assembly maps with coefficients in topological algebras}
We are going to use certain permanence properties to establish the equivalence between the Davis--L{\"u}ck and the Baum--Connes assembly maps with coefficients in $\cpt$. A key ingredient for us will be the Chabert--Echterhoff--Oyono-Oyono going-down mechanism \cite{CEO}. One obtains very refined permanence properties by localization of triangulated categories following Meyer--Nest \cite{MeyNes}, which might be helpful for further generalizations. 

Let $A$ be a separable and unital $C^*$-algebra, on which $G$ acts trivially. In this case the reduced crossed product $A\rtimes_r G$ simply becomes $A\prot C^*_r(G)$. In the sequel we denote $A\prot C^*_r(G)$ by $C^*_r(G,A)$. Recall that $C^*_r(G,A)$ is defined as a suitable completion of $C_c(G,A)=A[G]$, so that there is a canonical complex algebra homomorphism $\iota_A: A[G]\map C^*_r(G,A)$. Let $\iota_A: \bK_{A[G]}\map\bK_{C^*_r(G,A)}$ denote the induced map of $\K$-theory spectra.

\begin{rem} \label{Hextension}
By the naturality of the Loday assembly map there is a commutative diagram in the homotopy category of spectra:\beqn
\xymatrix{
BG_+\wedge\bK_{\tilde{A}}\ar[r]^{\mu^\L_{\tilde{A}}}\ar[d] &\bK_{\tilde{A}[G]}\ar[d] \\
BG_+\wedge\bK_\CC \ar[r]^{\mu^\L_\CC} & \bK_{\CC[G]},
}
\eeqn where $\tilde{A}$ is the unitization of $A$. Using excision we conclude that the induced map between the homotopy fibres $\mu^\L_A:BG_+\wedge\bK_A\map\bK_{A[G]}$ is the assembly map with coefficients in $A$. It follows that if both $\mu^\L_{\tilde{A}}$ and $\mu^\L_\CC$ are weak homotopy equivalences, then so is $\mu^\L_A$. The same argument allows us to construct the Loday assembly map with coefficients in any $\H$-unital $\QQ$-algebra, since such algebras also satisfy excision in algebraic $\K$-theory (see \cite{SusWod2}).
\end{rem}

\begin{defn} \label{loday}
We define the {\em algebraic reduced assembly map} $\mu^\alg_A: BG_+\wedge \bK_A \map \bK_{C^*_r(G,A)}$ by the following commutative diagram of spectra:
\beqn
\xymatrix{
&&\bK_{A[G]}\ar[d]^{\iota_A} \\
BG_+\wedge\bK_A \ar[rru]^{\mu^\L_A} \ar[rr]_{\mu^\alg_A} && \bK_{C^*_r(G,A)}.
}
\eeqn It follows that there is an exact triangle: \beq \label{defect}
\Cone(\mu^\L_A)\map\Cone(\mu^\alg_A)\map\Cone(\iota_A),\eeq where $\Cone$ denotes the mapping cone. 
\end{defn}

\begin{rem} \label{simplify1}
If $G$ is a discrete and torsion free group and $R$ is a unital, regular Noetherian $\QQ$-algebra, then $EG$ can be taken as a model for $E_\VC (G)$. In this case the Davis--L{\"u}ck assembly map $$(\mu^\DL_R)_*: \H^G_*(E_\VC (G);\bK_{G,R})\map \H^G_*(\textup{pt};\bK_{G,R})$$ is naturally equivalent to the Loday assembly map $(\mu^\L_R)_*$ (see Corollary 67 (ii) of \cite{LueRei}). 
\end{rem}

\noindent
We need a slight strengthening of the above observation.

\begin{lem} \label{simplify}
If $G$ is a discrete and torsion free group and $R$ is an $\H$-unital and $\K$-regular $\QQ$-algebra, then the Davis--L{\"u}ck assembly map $$(\mu^\DL_R)_*: \H^G_*(E_\VC (G);\bK_{G,R})\map \H^G_*(\textup{pt};\bK_{G,R})$$ is naturally equivalent to the Loday assembly map $(\mu^\L_R)_*:\H_*(BG;\bK_R)\map\K_*(R[G])$. 
\end{lem}

\begin{proof}
By excision in algebraic $\K$-theory we may assume that $R$ is unital. There is a commutative diagram relating the assembly maps $(\mu^\L_R)_*$ and $(\mu^\DL_R)_*$

\beqn
\xymatrix{
\H^G_*(E_\cF(G);\bK_{G,R})\cong\H_*(BG;\bK_R)\ar[rrd]^{(\mu^\L_R)_*}\ar[d]\\
\H^G_*(E_\VC(G);\bK_{G,R})\ar[rr]^{(\mu^\DL_R)_*} && \H^G_*(\textup{pt};\bK_{G,R})\cong\K_*(R[G]).
}
\eeqn Here the left vertical arrow is induced by the change of families morphism $\cF\subset\VC$. Indeed, by the universal property of $E_\VC(G)$ there is a map $E_\cF(G)\map E_\VC(G)$, which is well-defined up to $G$-homotopy. Therefore, it suffices to show that the left vertical arrow is an isomorphism. Now the argument in the proof of Proposition 70 of \cite{LueRei} applies. A careful inspection of the argument reveals that one only needs the vanishing of the Nil-terms for all group algebras $R[H]$, where $H\subset G$ is a finite subgroup. Since $G$ is assumed to be torsion free, this is guaranteed by the $\K$-regularity of $R$. 
\end{proof}

\begin{lem} \label{isomLem}
There are canonical isomorphisms

\begin{enumerate}
\item \label{isom1} $\H_*(BG;\bK_{G,\cpt})\overset{\cong}{\map}\H_*(BG;\bK^\top_{G,\cpt})$ induced by the change of theory morphism from the algebraic to topological $\K$-theory, where $\cpt$ is equipped with trivial $G$-action, and

\item \label{isom2} $\H_*(BG;\bK^\top_{G,\CC})\overset{\cong}{\map} \H_*(BG;\bK^\top_{G,\cpt})$ induced by the algebra homomorphism $\CC\map\cpt$.
\end{enumerate}
\end{lem}

\begin{proof}
For \eqref{isom1} we first reduce the problem to establishing the isomorphism $\H_*(BG;\bK_\cpt)\overset{\cong}{\map}\H_*(BG;\bK^\top_\cpt)$ by invoking Lemma \ref{simplify}. Since $\cpt$ is stable, the change of theory morphism $\bK_\cpt\map\bK^\top_\cpt$ is a homotopy equivalence whence the assertion follows.

\noindent
For \eqref{isom2} observe that, for every subgroup $H\subset G$, by construction $\H^G_*(G/H;\bK^\top_{G,\CC})\cong \KT_*(C^*_r(H))$ and $\H^G_*(G/H;\bK^\top_{G,\cpt})\cong\KT_*(C^*_r(H,\cpt))$. Due to $C^*$-stability of topological $\K$-theory, there is an induced isomorphism $\KT_*(C^*_r(H))\cong\KT_*(C^*_r(H,\cpt))$. The assertion now follows from Theorem 6.3 of \cite{DavLue}.
\end{proof}

Let us set $\K^\top_*(G;A)=\KK^G_*(\underline{E}G;A)$, where $\underline{E}G$ is the universal proper $G$-space, which is uniquely characterized up to $G$-homotopy. If $G$ is torsion free, one may take $\underline{E}G= EG$. The classifying space  $\underline{E}G$ may not be locally compact. So one defines $\KK^G_*(\underline{E}G;A) =\varinjlim_C \KK^G_*(C_0(X),A)$, where $C\subset \underline{E}G$ runs through the set of all $G$-compact subspaces canonically ordered by inclusion. Let $X$ be a locally compact proper $G$-space and let $A$ be a $G$-$C^*$-algebra (with not necessarily trivial $G$-action). For any $C^*$-algebra $B$ with trivial $G$-action, there is a canonical homomorphism $$\alpha_X:\KK_*^G(C_0(X),A)\otimes\K^\top_*(B)\map \KK^G_*(C_0(X),A\prot B),$$ which is constructed by first identifying $\K^\top_*(B)\cong\KK^G_*(\CC,B)$ and then applying Kasparov product $\otimes_\CC$. This map is compatible with inclusions of $G$-compact subspaces of $\underline{E}G$ and hence defines a morphism $$\alpha_G:\K^\top_*(G;A)\otimes\K^\top_*(B)\map \K^\top_*(G;A\prot B).$$

 In \cite{CEO} the authors define the following class $\cN_G$ of $G$-$C^*$-algebras: $A\in\cN_G$ if and only if the above morphism $\alpha_G$ is an isomorphism for every $C^*$-algebra $B$ with trivial $G$-action, such that $\K^\top_*(B)$ is torsion free. The class $\cN_G$ is fairly large; for instance, it contains all type I $C^*$-algebras and if $A\in \cN_G$ and $B$ is $\KK^G$-equivalent to $A$, then $B\in\cN_G$ (see Lemma 4.7 and Theorem 0.1 of \cite{CEO}). Clearly, $\cpt\in\cN_G$.

 \noindent
 Let $A$ be a $G$-$C^*$-algebra with trivial $G$-action. It follows from the commutative diagram in Section 1.6 (page 47) of \cite{BEL} that one has the following commutative diagram:
 \beq \label{DLcom}
\xymatrix{
\H^G_*(E_\cF (G);\bK_{G,A})\ar[d]_{}\ar[rr]^{(\mu^\DL_A)_*} && \K_*(A[G]) \ar[d]^{}\\
\H^G_*(E_\cF (G);\bK^\top_{G,A})\ar[rr]^{(\mu^\DL_A)_*} && \KT_*(C^*_r(G,A)),
}\eeq where the top (resp. bottom) horizontal arrow $(\mu^\DL_A)_*$ denotes the Davis--L{\"u}ck assembly map in algebraic (resp. topological) $\K$-theory induced by the pointed $G$-map $E_\cF(G)_+\map S^0$. Putting $A=\cpt$ with trivial $G$-action, where $G$ is a discrete and torsion free group, we get
\beq
\xymatrix{
\H_*(BG;\bK_{G,\cpt})\ar[d]_{\cong}^{}\ar[rr]^{(\mu^\DL_\cpt)_*} && \K_*(\cpt[G]) \ar[d]^{c(C^*_r(G,\cpt))_*\circ (\iota_\cpt)_*}\\
\H_*(BG;\bK^\top_{G,\cpt})\ar[rr]^{(\mu^\DL_\cpt)_*} && \KT_*( C^*_r(G,\cpt)),
}\eeq where the left vertical arrow is an isomorphism (cf. Lemma \ref{isomLem} part \eqref{isom1}). Now we use the fact that $\cpt$ is $\K$-regular \cite{RosComparison}, to replace $(\mu^\DL_\cpt)_*$ by $(\mu^\L_\cpt)_*$ (cf. Lemma \ref{simplify}) and obtain
 \beq \label{0}
\xymatrix{
&&\K_*(\cpt[G])\ar[d]^{(\iota_\cpt)_*}\\
\H_*(BG;\bK_\cpt)\ar[d]_{\cong}^{}\ar[rru]^{(\mu^\L_\cpt)_*} \ar[rr]^{(\mu^\alg_\cpt)_*}&& \K_*(C^*_r(G,\cpt)) \ar[d]^{c(C^*_r(G,\cpt))_*}\\
\H_*(BG;\bK^\top_{G,\cpt})\ar[rr]^{(\mu^\DL_\cpt)_*} && \KT_*( C^*_r(G,\cpt))
}\eeq The Hambleton--Pedersen result on the equivalence of $\mu^\DL_*$ with $\mu^\BC_*$ \cite{HamPed} is encapsulated in the following commutative diagram:
\beq \label{1}
\xymatrix{
\KT_*(BG)\cong\H_*(BG;\bK^\top_\CC)\ar[d]_{\cong}\ar[rr]^{\mu^\BC_*} && \KT_*(C^*_r(G)) \ar[d]^{\cong}\\
\H_*(BG;\bK^\top_{G,\CC})\ar[rr]^{\mu^\DL_*} && \KT_*( C^*_r(G)),
}\eeq where the vertical maps are isomorphisms. The naturality of the Davis--L{\"u}ck assembly map produces the following commutative diagram:
\beq \label{2}
\xymatrix{
\H_*(BG;\bK^\top_{G,\CC})\ar[d]_\cong\ar[rr]^{\mu^\DL_*} && \KT_*( C^*_r(G))\ar[d]^\cong\\
\H_*(BG;\bK^\top_{G,\cpt})\ar[rr]^{(\mu^\DL_\cpt)_*} && \KT_*(C^*_r(G,\cpt)),
}
\eeq where the left vertical arrow is an isomorphism (cf. Lemma \ref{isomLem} part \eqref{isom2}). From diagrams \eqref{1} and \eqref{2} we get the following commutative diagram:
\beq \label{3}
\xymatrix{
\H_*(BG;\bK^\top_\CC)\ar[d]_\cong\ar[rr]^{\mu^\BC_*} && \KT_*( C^*_r(G))\ar[d]^\cong\\
\H_*(BG;\bK^\top_{G,\cpt})\ar[rr]^{(\mu^\DL_\cpt)_*} && \KT_*(C^*_r(G,\cpt))
}\eeq We tensor the diagram \eqref{3} with $\KT_*(\cpt)$ over $\KT_*(\CC)$ and make the following simplifications:
\begin{enumerate}
\item We identify $\mu^\BC_*\otimes\id: \KT_*(BG)\otimes\KT_*(\cpt)\map\KT_*(C^*_r(G))\otimes\KT_*(\cpt)$ with the Baum--Connes assembly map with coefficients in the $G$-$C^*$-algebra $\cpt$ with trivial $G$-action $(\mu^\BC_\cpt)_*:\KT_*(BG;\cpt)\map\KT_*(C^*_r(G,\cpt))$ using Proposition 4.9 of \cite{CEO}.

\item Since $\KT_*(\cpt)$ is torsion free we use K{\"u}nneth formula in topological $\K$-theory \cite{RosSch} to identify $\KT_*(C^*_r(G,\cpt))\otimes\KT_*(\cpt)\cong\KT_*(C^*_r(G,\cpt))$. Now from Proposition \ref{domain} we conclude that $$\H_*(BG;\bK^\top_\cpt)\otimes \KT_*(\cpt)\cong\H_*(BG;\bK^\top_\cpt).$$ Observe that $\H_*(BG;\bK^\top_\cpt)\cong\KT_*(BG)$ via the homotopy equivalence $\bK^\top_\CC\overset{\sim}{\map} \bK^\top_\cpt$ induced by the $C^*$-algebra homomorphism $\CC\map\cpt$. Therefore the bottom horizontal arrow in the diagram \eqref{3} does not change.
\end{enumerate}

\noindent
Thus we have proved the following result:

\begin{prop} \label{BCDL}
There is a commutative diagram:
\beq \label{4}
\xymatrix{
\KT_*(BG;\cpt)\ar[d]_\cong\ar[rr]^{(\mu^\BC_\cpt)_*} && \KT_*( C^*_r(G,\cpt))\ar[d]^\cong \\
\H_*(BG;\bK^\top_{G,\cpt})\ar[rr]^{(\mu^\DL_\cpt)_*} && \KT_*(C^*_r(G,\cpt))
}\eeq expressing the equivalence of the Davis--L{\"u}ck assembly map and the Baum--Connes assembly map with coefficients in $\cpt$, when $G$ acts on it trivially.

\end{prop}

\noindent
Combining the diagram \eqref{0} with the above Proposition, we get

\begin{lem} \label{whitehead}
There is a commutative diagram:

\beq \label{5}
\xymatrix{
&& \K_*(\cpt[G])\ar[d]^{(\iota_\cpt)_*}\\
\H_*(BG;\bK_\cpt)\ar[rru]^{(\mu^\L_\cpt)_*}\ar[d]_{\cong}^{}\ar[rr]^{(\mu^\alg_\cpt)_*} && \K_*(C^*_r(G,\cpt)) \ar[d]^{c(C^*_r(G,\cpt))_*}_\cong\\
\KT_*(BG;\cpt)\ar[rr]^{(\mu^\BC_\cpt)_*} && \KT_*( C^*_r(G,\cpt))
}\eeq Observe that, by definition, $(\mu^\alg_\cpt)_*=(\iota_\cpt)_*\circ (\mu^\L_\cpt)_*$. This diagram expresses the equivalence between $(\mu^\alg_\cpt)_*$ and $(\mu^\BC_\cpt)_*$.

\end{lem}

\begin{thm} \label{cpt}
Let $G$ be a discrete and torsion free group. If $G$ satisfies the integral $\KT$-theoretic (resp. split $\KT$-theoretic) Novikov conjecture with $\CC$-coefficients, then $G$ and $\cpt$ satisfy the integral $\K$-theoretic (resp. split $\K$-theoretic) Novikov conjecture, i.e., the Loday assembly map $(\mu^\L_\cpt)_*:\H_*(BG;\bK_\cpt)\map\K_*(\cpt[G])$ is injective (resp. split injective).

\noindent
Furthermore, if $\mu^\BC_*$ is only rationally injective then so is $(\mu^\L_\cpt)_*$.
\end{thm}

\begin{proof}
For any $G$-$C^*$-algebra $A$ and any other $C^*$-algebra $B$ with a trivial $G$-action the authors of \cite{CEO} construct the following commutative diagram in Proposition 4.9:

\beqn
\xymatrix{
\KT_*(BG;A)\otimes \KT_*(B)\ar[d]_{\alpha_G}\ar[rr]^{(\mu^\BC_{A}\otimes\id)_*} && \KT_*(A\rtimes_r G)\otimes \KT_*(B)\ar[d] \\
\KT_*(BG;A\prot B) \ar[rr]^{(\mu^\BC_{A\prot B})_*} && \KT_*( (A\prot B)\rtimes_r G)
}
\eeqn where the right vertical arrow is the K{\"u}nneth map in topological $\K$-theory. Putting $A=\CC$ and $B=\cpt$ we get
\beqn
\xymatrix{
\KT_*(BG)\otimes \KT_*(\cpt)\ar[d]_{\alpha_G}^{\cong}\ar[rr]^{(\mu^\BC\otimes\id)_*} && \KT_*(C^*_r(G))\otimes \KT_*(\cpt)\ar[d]^{\cong}\\
\KT_*(BG;\cpt) \ar[rr]^{(\mu^\BC_{\cpt})_*} && \KT_*(C^*_r(G,\cpt))
}
\eeqn Since $\KT_*(\cpt)$ is torsion free and $\cpt$ is in the bootstrap class both vertical arrows are isomorphisms. Using flatness of $\KT_*(\cpt)$ we conclude that if $\mu^\BC_*$ is (split) injective, then so is $(\mu^\BC\otimes\id)_*\cong (\mu^\BC_\cpt)_*$. The assertion now follows from the previous Lemma.

The statement about the rational injectivity of $(\mu^\L_\cpt)_*$ is obvious.
\end{proof}

\section{Algebraic $\K$-theory of certain group algebras}

Recall that the homotopy cofibre of the Loday assembly map $\mu^\L_R: BG_+\wedge\bK_R\map \bK_{R[G]}$ is defined to be the {\em Whitehead spectrum of $G$ over $R$}, and its homotopy groups are called the {\em Whitehead groups of $G$ over $R$}. By excision in algebraic $\K$-theory this notion carries over to $\H$-unital coefficient $\QQ$-algebras (see Remark \ref{Hextension}).

\begin{lem} \label{obstruction}
Let $G$ satisfy the Baum--Connes conjecture with $\CC$-coefficients. Then the obstruction to the algebra homomorphism $\cpt[G]\map C^*_r(G,\cpt)$ inducing a weak homotopy equivalence at the level of nonconnective algebraic $\K$-theory spectra lies in the Whitehead spectrum of $G$ over $\cpt$ (up to a shift).
\end{lem}

\begin{proof}
Since the bijectivity of $\mu^\BC_*$ implies that of $(\mu^\BC_\cpt)_*$ (see Corollary 5.2 of \cite{CEO}) and the algebraic reduced assembly map $(\mu^\alg_\cpt)_*$ at the level of homotopy groups agrees with $(\mu^\BC_\cpt)_*$ (see Lemma \ref{whitehead}), the assertion follows from \eqref{defect}.
\end{proof}

\begin{thm} \label{mySW}
Let $G$ be a discrete and torsion free Gromov hyperbolic group and let $G$ act on $\cpt$ trivially. Then the canonical algebra homomorphism $\iota_\cpt: \cpt[G]\map C^*_r(G,\cpt)$ induces a weak homotopy equivalence between their nonconnective algebraic $\K$-theory spectra, and the Whitehead groups of $G$ over $\cpt$ vanish.
\end{thm}

 \begin{proof}
Under the assumptions on the group $G$ it is known that it satisfies the Baum--Connes conjecture with $\CC$-coefficients \cite{MinYu}. By the previous Lemma it suffices to show that the Whitehead groups of $G$ over $\cpt$ vanish. A result of Bartels--L{\"u}ck--Reich says that all Gromov hyperbolic groups satisfy the Farrell--Jones isomorphism conjecture in algebraic $\K$-theory for every associative and unital ring $R$ (see Corollary 1.2 of \cite{BLR}); more precisely, the authors prove that for torsion free groups $\K_n(R[G])\cong\H_n(BG;\bK_R) \oplus\left( \oplus_I (\NK_n(R)\oplus \NK_n(R))\right)$, where $I$ denotes the set of conjugacy classes of maximal infinite cyclic subgroups of $G$. Using the naturality of the decomposition (in $R$), Remark \ref{Hextension} and excision in $\NK$-theory, one concludes
\beq \label{BR}
 \K_n(\cpt[G])\cong\H_n(BG;\bK_\cpt) \oplus\left( \oplus_I (\NK_n(\cpt)\oplus \NK_n(\cpt))\right).
 \eeq It is known that $\NK_n(\cpt)$ vanishes for all $n$, since $\cpt$ is a stable $C^*$-algebra (see Theorem 3.4 of \cite{RosComparison}). It follows that the Whitehead groups of $G$ over $\cpt$ vanish.
 \end{proof}

\begin{rem}
Since the $C^*$-algebra $C^*_r(G,\cpt)$ is stable, its nonconnective algebraic $\K$-theory is the same as its topological $\K$-theory, which in turn is the same as its topological $\K$-homology, i.e.,  $\K_*(\cpt[G])\simeq \KT_*(BG)$.
\end{rem}

\subsection{Algebraic $\K$-theory of $\S[G]$}
For any $p\geqslant 1$ let $\S_p$ denote the ring of operators of Schatten $p$-class, i.e., $T\in B(H)$ is an element of $\S_p$ if and only if $\textup{tr}((T^*T)^{\frac{p}{2}})<\infty$ (see \cite{Schatten} for further details). The algebra of Schatten class operators is defined to be $\S=\cup_{p\geqslant 1} \S_p$. There is a canonical sequence of $\CC$-algebra inclusions $\S\subset \cpt \subset B(H)$. It follows from Lemma 8.2.3 and Theorem 8.2.5 of \cite{CorTho1} that $\S[G]$ is $\H$-unital for any discrete group $G$.

 \begin{conj}[Yu \cite{GuoliangNovikov}] \label{YuConj}
 For any discrete group $G$ the canonical algebra homomorphism $i:\S[G]\cong\CC[G]\otimes_\CC \S\map C^*_r(G)\prot \cpt = C^*_r(G,\cpt)$ induces an isomorphism between their algebraic $\K$-theory groups $$i_*: \K_n(\S[G])\map\K_n( C^*_r(G,\cpt)).$$
 \end{conj}

The algebra homomorphism $i:\S[G]\map C^*_r(G,\cpt)$ can be factorized as $\S[G]\map\cpt[G]\map C^*_r(G,\cpt)$. Now one can investigate the analogue of Yu's conjecture for these two homomorphisms separately. The Farrell--Jones isomorphism conjecture in algebraic $\K$-theory implies the analogue of Yu's conjecture for $\S[G]\map\cpt[G]$. Theorem \ref{mySW} above gives an affirmative answer to the analogue of Yu's conjecture for $\cpt[G]\map C^*_r(G,\cpt)$ when the group $G$ is torsion free Gromov hyperbolic (note that such groups are known to satisfy both the Farrell--Jones isomorphism conjecture in algebraic $\K$-theory and the Baum--Connes conjecture \cite{MinYu,Lafforgue}).

Observe that $\pi_*(\bK_\cpt)$ is Bott $2$-periodic (due to its identification with topological $\K$-theory). In fact, $\pi_*(\bK_\cpt)$ is $\ZZ$ if $*$ is even, and $\{0\}$ if $*$ is odd. Since $\S$ is a sub-harmonic ideal, the same conclusion holds for the algebraic $\K$-theory of $\S$ (see Theorem 8.2.5 of \cite{CorTho1} and Corollary 4.2 of \cite{CorTar}). An easy inspection reveals that the canonical inclusion $\S\map\cpt$ induces a weak homotopy equivalence $\bK_\S\map\bK_\cpt$. It is also shown in Theorem 8.2.5 of \cite{CorTho1} that $\S$ is $\K$-regular, and we already know that it is $\H$-unital. Therefore, the Loday assembly map can be defined with coefficients in $\S$ (see Remark \ref{Hextension}). As before, for any discrete and torsion free group $G$, using Lemma \ref{simplify} we identify the Davis--L{\"u}ck assembly map in $\K$-theory with the Loday assembly map $(\mu^\L_\S)_*:\H_*(BG;\bK_\S)\map\K_*(\S[G])$. Excision and $\K$-regularity of $\S$ enable us to conclude that the Whitehead groups of any torsion free Gromov hyperbolic group over $\S$ vanish. Using the naturality of the Loday assembly map, once again we have the following commutative diagram:
\beqn
\xymatrix{
\H_*(BG;\bK_\S)\ar[r]\ar[d]_\cong & \K_*(\S[G])\ar[d] \\
\H_*(BG;\bK_\cpt)\ar[r] & \K_*(\cpt[G]),
}\eeqn where the vertical arrows are induced by $\S\map\cpt$. Since the map $\bK_\S\map\bK_\cpt$ is a weak homotopy equivalence the left vertical arrow is an isomorphism. Now Theorem \ref{cpt} implies

\begin{thm} \label{overS}
Let $G$ be a discrete and torsion free group. If $G$ satisfies the integral $\KT$-theoretic (resp. split $\KT$-theoretic) Novikov conjecture with $\CC$-coefficients, then $G$ and $\S$ satisfy the integral $\K$-theoretic (resp. split $\K$-theoretic) Novikov conjecture, i.e., the Loday assembly map $(\mu^\L_\S)_*:\H_*(BG;\bK_\S)\map\K_*(\S[G])$ is injective (resp. split injective).
\end{thm}

\begin{rem}
The above Theorem is an integral $\K$-theoretic statement, whereas the main result of Yu in \cite{GuoliangNovikov} proves the rational injectivity of $(\mu^\L_\S)_*$ for all discrete groups with no further assumptions (see also \cite{CorTar}).
\end{rem}

\begin{rem}
Summing up one has the following sequence of implications (inj. = injective):
\beqn
\text{$\mu^\BC_*$ (split) inj. $\Leftrightarrow$ $(\mu^\BC_\cpt)_*$ (split) inj. $\Rightarrow$ $(\mu^\L_\cpt)_*$ (split) inj. $\Rightarrow$ $(\mu^\L_\S)_*$ (split) inj.,}
\eeqn where the first equivalence is due to \cite{CEO}. 
\end{rem}

\section{$\K$-theoretic Novikov conjecture with finite coefficients}
Browder initiated the study of $\K$-theory with finite coefficients \cite{Browder}. For a pointed space $X$ one obtains its homotopy groups with coefficients in $\ZZ/n$ by replacing the spheres $S^i$ in the definition of homotopy groups $\pi_i(X)=[S^1,X]$ by the Moore space $M_n^i:= S^{i-1}\cup_n e^i$, i.e., an $i$-cell attached to $S^{i-1}$ by a map $S^{i-1}\map S^{i-1}$ of degree $n$. This enables us to construct {\em $\K$-theory with finite coefficients} \cite{Browder}, which enjoys many good functorial properties like ordinary $\K$-theory. For a recent survey of the theory see \cite{Neisendorfer}. Another possibility of introducing coefficients in $\K$-theory is the following: Let $\SS(\ZZ/n)$ denote the mod-$n$ Moore spectrum, i.e., the cofibre of a degree $n$ map $\SS\map\SS$ between the sphere spectra. For any $C^*$-algebra $A$, define $\K_*(A,\ZZ/n) := \pi_*(\bK_A\wedge\SS(\ZZ/n))$ (resp. $\KT_*(A,\ZZ/n) := \pi_*(\bK^\top_A\wedge\SS(\ZZ/n))$). The two possible definitions will agree connectively. Karoubi studied algebraic and topological $\K$-theory of Banach algebras with finite coefficients and obtained some striking results in \cite{KarFinCoe}. It is known that if  $A=\CC$ or $\cpt$, the comparison map with finite coefficients $\K_*(A,\ZZ/n)\map\K^\top_*(A,\ZZ/n)$ is an isomorphism for all $n\geqslant 2$ and $*\geqslant 0$ (see Theorem 3.11 of \cite{RosComparison}).

\begin{lem} \label{AHSS}
Let $\E$ be a connective spectrum and $\E\map\F$ be a map of spectra such that $\pi_*(\E)\map\pi_*(\F)$ is an isomorphism for all $*\geqslant 0$. Then the induced map $\pi_*(\E,\ZZ/n)\map\pi_*(\F,\ZZ/n)$ is an isomorphism for all $*\geqslant 1$ and $\pi_0(\E,\ZZ/n)\map\pi_0(\F,\ZZ/n)$ is a monomorphism. The map $\pi_0(\E,\ZZ/n)\map\pi_0(\F,\ZZ/n)$ is an isomorphism if $\pi_{-1}(\F)=\{0\}$.
\end{lem}

\begin{proof}
Due to the naturality of the Universal Coefficient Theorem for homotopy groups with finite coefficients (for spectra), we have the following commutative diagram 
\beqn
\xymatrix{
0\ar[r] & \pi_*(\E)\otimes\ZZ/n \ar[r]\ar[d] & \pi_*(\E,\ZZ/n) \ar[r]\ar[d] &\Tor(\pi_{*-1}(\E),\ZZ/n)\ar[r]\ar[d] & 0\\
0 \ar[r] & \pi_*(\F)\otimes\ZZ/n \ar[r] & \pi_*(\F,\ZZ/n) \ar[r] & \Tor(\pi_{*-1}(\F),\ZZ/n) \ar[r] & 0.
}\eeqn The assertion for $*\geqslant 1$ now follows from the Short Five Lemma. Use the Snake Lemma for the injectivity (resp. the bijectivity) at the level of $\pi_0$ (resp. when $\pi_{-1}(\F)=\{0\}$).
\end{proof}

Let $\E$ be a (left) module spectrum over the orbit category $\Or(G)$. Given any $G$-CW complex $X$ one can construct an ordinary spectrum via the {\em balanced smash product} construction $$X^H_+\wedge_{\Or(G)} \E(G/H)=\coprod_{G/H} X^H_+\wedge \E(G/H)/\{\sim\},$$ where $\sim$ is the equivalence relation generated by $(x\phi,y)\sim (x,\phi y)$ for all $x\in X^K_+$, $y\in\E(G/H)$ and $\phi:G/H\map G/K$ (as explained in Section 5 of \cite{DavLue} in a slightly different notation). Let $E(G)$ be a classifying $G$-space for $G$ with respect to any family of subgroups. Then the pointed $G$-projection $E(G)_+\map S^0=(G/G)_+$ gives rise to a map of spectra \beq \label{SpeAss} E(G)^H_+\wedge_{\Or(G)}\E(G/H)\map \E(G/G).\eeq Applying homotopy groups we obtain the assembly maps in the $G$-homology theory defined by the $\Or(G)$-module spectrum $\E$. 

\begin{defn} \label{FinAss}
By applying the functor homotopy groups with finite coefficients to the above map \eqref{SpeAss}, we obtain the {\em assembly map with finite coefficients in the $G$-homology theory defined by $\E$} $$\H^G_*((E(G);\E),\ZZ/n)\map\H^G_*((S^0;\E),\ZZ/n).$$
\end{defn}

\noindent
Since $\CC$ is a regular noetherian ring containing $\QQ$, by Proposition 70 of \cite{LueRei} we conclude that the pointed $G$-projection $E_\cF(G)_+\map S^0$ induces $\mu^\DL_*$ in $\K$-theory (resp. $\KT$-theory) when $\E$ is the $\Or(G)$-module $\K$-theory spectrum (resp. $\KT$-theory spectrum).

\begin{thm} \label{overC}
Let $G$ be a discrete and torsion free group, such that the Baum--Connes assembly map with $\CC$-coefficients $\mu^\BC_*:\KT_*(G;\CC)\map\KT_*(C^*_r(G))$ is split injective. Assume, in addition, that the $\ZZ/n$-homology of $BG$ is concentrated in even degrees. Then $G$ satisfies the $\K$-theoretic Novikov conjecture with $\ZZ/n$-coefficients over $\CC$ and $\overline{\QQ}$, i.e., the following two Loday assembly maps with finite coefficients are injective: $$\H_*((BG;\bK_{\CC}),\ZZ/n)\map \K_*({\CC}[G],\ZZ/n) \text{  \quad and \quad  }\H_*((BG;\bK_{\overline{\QQ}}),\ZZ/n)\map \K_*(\overline{\QQ}[G],\ZZ/n).$$
\end{thm}

\begin{proof}
Let us first consider the case over $\CC$. Since $\H_*((BG;\bK_{\CC}),\ZZ/n)$ is trivial for $*<0$, we may concentrate only on nonnegative degrees. From Definition \ref{loday} and the commutative diagram \eqref{DLcom} we get
\beq \label{main}
\xymatrix{
&&\K_*(\CC[G])\ar[d]^{(\iota_\CC)_*}\\
\H_*(BG;\bK_\CC)\ar[d]_{c(\CC)_*}\ar[rr]^{\mu^\alg_*} \ar[rru]^{\mu^\L_*} && \K_*( C^*_r(G)) \ar[d]^{c(C^*_r(G))_*}\\
\H_*(BG;\bK^\top_\CC)\ar[rr]^{\mu^\DL_*} && \K^\top_*( C^*_r(G)),
}\eeq The bottom horizontal arrow $\mu^\DL_*$ can be replaced by $\mu^\BC_*$ (thanks to the Hambleton--Pedersen Theorem $\mu^\DL_*\cong\mu^\BC_*$ \cite{HamPed}), which due to our assumption is split injective. It is clear how to incorporate $\ZZ/n$-coefficients into the diagram to obtain the corresponding assembly maps with finite coefficients (see Definition \ref{FinAss} above), since all the arrows in the above diagram have a homotopy theoretic origin; in other words, the diagram is obtained by applying homotopy groups with finite coefficients to the corresponding diagram of spectra. We are going to show that the left vertical arrow and the bottom horizontal arrow in \eqref{main} are injective with $\ZZ/n$-coefficients, which is enough to prove the result.

Observe that the Davis--L{\"u}ck assembly map $\mu^\DL_*$ in topological $\K$-theory is induced by the $G$-projection $E_\cF(G)_+\map S^0$. Thanks to the naturality (in the space variable) of the Universal Coefficient Theorem of homotopy groups with finite coefficients we have the following commutative diagram:
\beqn
\xymatrix{
0\ar[r] & \H_*(BG;\bK^\top_\CC)\otimes\ZZ/n \ar[r]\ar[d] & \H_*((BG;\bK^\top_\CC),\ZZ/n) \ar[r]\ar[d] &\Tor(\H_{*-1}(BG;\bK^\top_\CC),\ZZ/n)\ar[r]\ar[d] & 0\\
0 \ar[r] & \KT_*(C^*_r(G))\otimes\ZZ/n \ar[r] & \KT_*(C^*_r(G),\ZZ/n) \ar[r] & \Tor(\KT_{*-1}(C^*_r(G),\ZZ/n) \ar[r] & 0,
}
\eeqn where the vertical arrows are induced by the Davis--L{\"u}ck assembly maps. Notice that the two extremal vertical arrows are split injective. Now the Snake Lemma enables us to conclude that the middle vertical arrow (which is the bottom horizontal arrow in \eqref{main} with $\ZZ/n$-coefficients) is injective. The injectivity of the left vertical arrow in \eqref{main} with $\ZZ/n$-coefficients follows from a celebrated result of Suslin \cite{Suslin,SusLoc}, which says that $c(\CC):\bK_\CC\map\bK^\top_\CC$ induces an isomorphism $\K_*(\CC,\ZZ/n)\map\KT_*(\CC,\ZZ/n)$ in nonnegative degrees (and a split monomorphism otherwise). Indeed, the map $\bK_\CC\map\bK^\top_\CC$ can be factorized as $\bK_\CC\map\bK^\top_\CC\langle 0\rangle\map\bK^\top_\CC$, where $\bK^\top_\CC\langle 0\rangle$ denotes the connective cover of $\bK^\top_\CC$. It follows that the map $\H_*((BG;\bK_\CC),\ZZ/n)=\pi_*(BG_+\wedge\bK_\CC,\ZZ/n)\map\pi_*(BG_+\wedge\bK^\top_\CC\langle 0\rangle,\ZZ/n)$ is an isomorphism. The proof is now completed by invoking Lemma \ref{AHSS} and examining the morphism between the Atiyah--Hirzebruch spectral sequences for $BG$ induced by the map $\bK^\top_\CC\langle 0\rangle\wedge\SS(\ZZ/n)\map\bK^\top_\CC\wedge \SS(\ZZ/n)$, where $\SS(\ZZ/n)$ denotes the mod-$n$ Moore spectrum. Observe that under the extra hypothesis on the $\ZZ/n$-homology of $BG$, both Atiyah--Hirzebruch spectral sequences collapse at the $E^2$-page.

The assertion over $\overline{\QQ}$ now follows from another Theorem of Suslin that says: the field extension $\overline{\QQ}\hookrightarrow\CC$ induces an isomorphism in algebraic $\K$-theory with finite coefficients \cite{Suslin,SusLoc}.
\end{proof}

\begin{rem}
In the above Theorem \ref{overC}, the condition that the $\ZZ/n$-homology of $BG$ be concentrated in even degrees is satisfied, for instance, if the integral homology of $BG$ is concentrated in even degrees and it contains no $n$-torsion. In general, it suffices to impose conditions on $G$ that ensure the collapsing of the Atiyah--Hirzebruch spectral sequences $E^2_{p,q}=H_p(BG;\KT\langle 0\rangle_q(\CC,\ZZ/n))$ and $E^2_{p,q}=H_p(BG;\KT_q(\CC,\ZZ/n))$ at the $E^2$-pages. Here $\KT\langle 0\rangle_q(\CC,\ZZ/n)$ denotes the connective topological $\K$-theory groups of $\CC$ with $\ZZ/n$-coefficients. A simple computation reveals that $\KT_q(\CC,\ZZ/n)\simeq \ZZ/n$ if $q$ is even and $\{0\}$ otherwise.
\end{rem}

\begin{rem} \label{Gen}
For an arbitrary discrete group $G$ (not necessarily torsion free), there is a commutative diagram
\beqn
\xymatrix{
\H^G_*((E_\cF(G);\bK_{G,\CC}),\ZZ/n)\ar[r]\ar[d] & \K_*(\CC[G],\ZZ/n) \ar[d]\\
\H^G_*((E_\cF(G);\bK^\top_{G,\CC}),\ZZ/n)\ar[r] & \KT_*(C^*_r(G),\ZZ/n),
}
\eeqn where the horizontal arrows are induced by the assembly maps with finite coefficients. The split injectivity of $\mu^\BC_*$ will imply the injectivity of the bottom horizontal arrow. Therefore, in order to generalize Theorem \ref{overC} to groups containing torsion, i.e., to establish the injectivity part of the Farrell--Jones isomorphism conjecture in $\K$-theory with finite coefficients, one must look for conditions under which the left vertical arrow is injective. The $p$-chain spectral sequence \cite{DavLueSS} is relevant for this question.
\end{rem}


\bibliographystyle{abbrv}

\bibliography{/home/tubai/Professional/math/MasterBib/bibliography}

\vspace{3mm}

\end{document}